\documentclass[11pt,twoside]{amsart}
\usepackage{bbm}
\usepackage{graphicx,color}
\usepackage{enumerate}
\usepackage[all]{xy}
\usepackage[hyperindex=true,plainpages=false,colorlinks=false,pdfpagelabels]{hyperref}
\usepackage{verbatim}
\usepackage[hyperindex=true,plainpages=false,colorlinks=false,pdfpagelabels]{hyperref}

\theoremstyle{plain}
\newtheorem{The}{Theorem}
\newtheorem*{The*}{Theorem}
\newtheorem{Pro}{Proposition}
\newtheorem{Lem}{Lemma}
\newtheorem{Cor}{Corollary}
\newtheorem*{Cor*}{Corollary}

\theoremstyle{definition}
\newtheorem*{Def}{Definition}
\newtheorem{Rem}{Remark}
\newtheorem{Exa}{Example}
\newtheorem*{Rem*}{Remark}

\numberwithin{equation}{section}

\DeclareMathOperator{\del}{\partial}

\newcommand{\R}{\mathbb{R}}
\newcommand{\C}{\mathbb{C}}
\newcommand{\N}{\mathbb{N}}
\newcommand{\Z}{\mathbb{Z}}

\begin{document}

\title[Constrained Willmore Tori and elastic curves]{Constrained Willmore tori and elastic curves in $2-$dimensional space forms}

\author{Lynn Heller}

\address{Lynn Heller \\  Institut f\"ur Mathematik\\
Universit{\"a}t T\"ubingen\\ Auf der Morgenstelle
10\\ 72076 T\"ubingen\\ Germany
}

\email{lynn-jing.heller@uni-tuebingen.de}

\subjclass{53A04, 53A05, 53A30, 37K15}

\date{\today}

\thanks{Author supported by SFB/Transregio 71}

\begin{abstract} 
In this paper we consider two  special classes of constrained Willmore tori in the $3-$sphere.  The first class is given by the rotation of  closed elastic curves in the upper half plane - viewed as the hyperbolic plane - around the $x-$axis. The second is given as the preimage of  closed constrained elastic curves, i.e., elastic curves with enclosed area constraint,  in the round $2-$sphere under the Hopf fibration.
We show that all conformal types can be isometrically immersed into $S^3$ as constrained Willmore (Hopf) tori
and explicitly parametrize all constrained elastic curves in $H^2$ and $S^2$ in terms of the Weierstrass elliptic functions. Further, we determine the closing condition for the curves and compute the Willmore energy and the conformal type of the resulting tori.
 \end{abstract}

\maketitle


\section{Introduction}
\label{sec:intro}
Let $f: M \rightarrow S^3$ be a conformally immersed compact surface. It is called 
constrained Willmore, if it is a critical point of the Willmore energy $ \int_{M}(H^2+1)dA $ under conformal variations.
The minimizer of the Willmore energy for a fixed conformal class can be viewed as the optimal realization of the underlying Riemann surface in three space.
Such a minimizer exists for $M$, see \cite{KS}, if the underlying conformal class provides  a immersion to $S^3$ with Willmore energy below $8 \pi$. Further, the minimizer is smooth and constrained Willmore. It is an open question whether the infimum of the Willmore energy is below $8 \pi$ for every conformal class. \\

The global minimizer of the Willmore energy in the class of tori is the Clifford torus, see \cite{MN}. Further,  \cite{NS}  have shown that the homogenous tori $T_r$ are minimizers of their respective conformal classes near the Clifford torus.
For rectangular conformal classes the minimizers are conjectured to be the $2-$lobed tori of revolution, which have constant mean curvature in $S^3$, see figure \ref{2lobe}. \begin{figure}[htbp]\label{2lobe}
  \centering
  \includegraphics[scale=0.31]{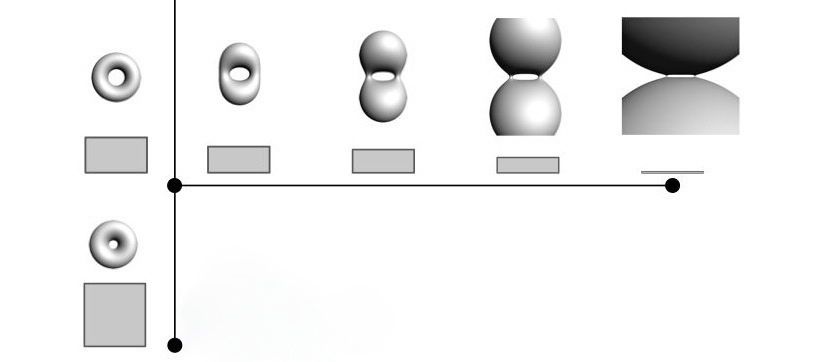}
  \caption{Embedded two lobed CMC tori of revolution in $S^3.$ (by Nick Schmitt)}
\end{figure}
The Willmore energy of this family increases monotonically with the conformal type, see \cite{KSS}, and converges to $8\pi$. The limiting surface is a double covering of a geodesic sphere. Thus the minimizer of the Willmore energy for tori with prescribed rectangular conformal class exists by \cite{KS}. 
Tori of revolution can be constructed by rotation of a closed curve in the upper half plane around the $x-$axis. The torus is constrained Willmore if and only if the curve is elastic in the upper half plane viewed as $H^2.$
Since \cite{AL} have shown  that all embedded CMC tori are rotational, the pictured tori are the minimizers of the Willmore energy in their respective conformal classes restricted to CMC tori. For non rectangular  conformal classes no candidates for the minimizers are known in the literature, since tori of revolution are always of rectangular conformal types.\\

First examples of Willmore tori, which are not minimal in a space form were found  by \cite{P} in the class of Hopf tori. These are  given by the  preimage of closed curves in $S^2$ under the  Hopf fibration.  The torus is (constrained) Willmore, if and only if the corresponding curve is (constrained) elastic, i.e., critical points of the energy functional with prescribed length and enclosed area. In contrast to tori of revolution \cite{P} shows that all conformal classes can be realized algebraically as Hopf tori. \\

In the literature there exists an alternative notion of constrained Willmore surfaces. These are critical points of the Willmore functional with prescribed enclosed volume and surface area  (Helfrich model). Since Hopf tori are flat and the mean curvature of the torus is simply the geodesic curvature of the curve in $S^2$, constrained Willmore Hopf tori are constrained Willmore in both sense. \\

In this paper we study the two classes of constrained Willmore tori which comes from closed elastic curves in $H^2$ and closed (constrained) elastic curves in $S^2.$ We first show that every conformal class can be realized as a constrained Willmore (Hopf) torus via the direct method of calculus of variations.  This generalizes the result by \cite{P}. Then we derive explicit formulas for (constrained) elastic curves in $2-$dimensional space forms. By viewing $H^2$ and $S^2$ as subsets of $\C P^1$, we define the Schwarzian derivative $q$ as a M\"obius invariant of a curve $\gamma$ in $\C P^1$. The curve $\gamma$ is constrained elastic if and only if its Schwarzian derivative is stationary under the first order KdV flow. Thus $q$ is generically given in terms of a Weierstrass $\wp-$function defined on a torus $\C / \Gamma,$ which plays the role of a spectral curve in our setting. We compute the closing conditions for the curves and show that every constrained elastic curve is isospectral to an elastic curve. 
Then we give formulas for the Willmore energy and the conformal type of the resulting torus.\\

In their paper \cite{LS} Langer and Singer constructed elastic curves in $S^2$ and $H^2$ without the enclosed area constraint. Our result is a generalization of this and uses the Schwarzian derivative instead of the geodesic curvature of the curve. \\

The author wants to thank Christoph Bohle and Cheikh B. Ndiaye for helpful discussions and Nick Schmitt for making the figures throughout the paper.

\section{Equivariant tori in the $3-$sphere}
\noindent
We consider $S^3 \subset \C^2.$  
\begin{Def}
A map $f: \C \rightarrow S^3 $  is called $\R-$equivariant, if there exist group homomorphisms 
\begin{equation*}
\begin{split}
&M: \R \rightarrow \{\text{M\"obius transformations of }S^3 \}, t \mapsto M_t,\\
&\tilde M: \R \rightarrow \{\text{conformal transformations of } \C\}, t \mapsto \tilde M_t,
\end{split}
\end{equation*}
such that 
$$f \circ \tilde M_t = M_t \circ f,  \text{ for all } t.$$
\end{Def}
\noindent
If $f$ is doubly periodic with respect to a lattice $\Gamma \subset \C,$ then $f$ is a torus and the following proposition holds.
\begin{Pro}\label{equivariant}
Let $f: T^2 \cong \C/\Gamma \rightarrow S^3$ be a equivariant conformal immersion. Then there exist a holomorphic coordinate $z = x+iy$ of $T^2$ together with $m , n \in \N$ and $gcd(m,n) = 1$ such that 
$$f(x,y) = \begin{pmatrix} e^{i m x}&0\\ 0& e^{i n x}\end{pmatrix} f(0,y),$$
up to isometries of $S^3$ and the identification of $S^3$ with $SU(2).$
The  curve $\gamma(y) := f(0,y)$ (not necessarily closed) is called the profile curve of the surface.   
\end{Pro} 
\noindent
In this paper we only consider two very special cases of equivariant tori, namely the case of tori of revolution ($m= 1, n= 0$) and Hopf tori ($m= n = 1$). 

\begin{Def} Let $M$ be a compact and oriented surface and let $f: M \rightarrow S^3$ be an immersion into the round sphere. The \emph{Willmore energy} of $f$ is defined to be
$$\mathcal W (f) = \int_ M( H^2 + 1 ) dA,$$
where $H$ is the mean curvature of $f$ and $dA $ is  induced volume form. \\

\noindent
A conformal immersion $f: M \rightarrow S^3$ is called {\it Willmore}, if it is a critical point of the Willmore energy $W$ under all variations and it is called {\it constrained Willmore}, if it is a critical point of $W$ under conformal variations, see \cite{BoPetP} and \cite{S}.\\
\end{Def}

It is shown in \cite{LS1} that  the Willmore functional reduces to the energy functional $\int_{\gamma} \kappa^2 ds$ for surfaces of revolution, where $\kappa$ is the curvature of the profile curve in the hyperbolic plane, and $s$ is the arc length parameter. The conformal type of the torus is determined by the length of the curve in $H^2.$
Further, \cite{P} shows that the Willmore energy for  a Hopf torus reduces to the generalized energy functional $\int_{\gamma} (\kappa^2 + 1)ds$ of the corresponding curve in $S^2.$ In particular, the mean curvature of Hopf tori satisfies $H= \kappa$ and by construction the Gau\ss ian curvature is zero. The conformal type of the torus is determined by the length and enclosed area of the curve. Thus by the principle of symmetric criticality \cite{Palais}, i.e., the critical symmetric points are the symmetric critical points, a surface of revolution is constrained Willmore if and only if its profile curve is elastic in $H^2$ and a Hopf torus is constrained Willmore, if its profile curve is a critical point of the energy functional with prescribed length and enclosed area.

\begin{Def}
Let $\gamma$ be an arc length parametrized closed curve in a $2-$dimen-sional space form and $\kappa$ its geodesic curvature.
The curve is called {\it constrained elastic}, if it is a critical point of the energy functional $\int_{\gamma} \kappa^2 ds$ with fixed length and enclosed area.
\end{Def}

\begin{Pro}[\cite{BoPetP}]
Let $\gamma$ be an arclength parametrized curve into a $2-$di-mensional space form of constant curvature $G$ and let $\kappa$ be its geodesic curvature in the space form. 
The Euler-Lagrange equation for a constrained elastic curve is:
\begin{equation}\label{secondorder}
\kappa'' + \frac{1}{2}\kappa^3 + (\mu + G)\kappa + \lambda = 0,
\end{equation}
for real parameters $\mu$ and $\lambda.$
\end{Pro}

This equation is the well known  stationary first order modified Korteweg-de-Vries equation. The real parameters $\mu$ and $\lambda$ are the length and  respectively the enclosed area constraint for a closed curve.  A solution to $\mu = \lambda = 0$ is a free elastic curve in the space form of curvature $G.$ By multiplying the equation with $2 \kappa'$ the equation can be integrated once and yields
\begin{equation}\label{P4}
(\kappa')^2  =  -\frac{1}{4}\kappa^4 - (\mu + G)\kappa^2 - 2 \lambda \kappa - \nu.
\end{equation}
Here $\nu$ is a real integration constant. We denote the negative of the polynomial on the right hand side by $P_4$, i.e., $$P_4:= \frac{1}{4}x^4 + (\mu + G)x^2 + 2 \lambda x + \nu.$$ 

\begin{The}
For given real numbers $L_0$ and $A_0$ satisfying the isoperimetric inequality on $S^2$
$$L_0^2 - 4 \pi A_0  + A_0^2 \geq 0,$$ 
there exist a smooth constrained elastic curve in $S^2$ minimizing the energy $\mathcal E(\gamma) = \int (\kappa^2 + 1) ds$ with length $L(\gamma) = L_0$ and enclosed area $A(\gamma) = A_0.$
\end{The}
\begin{Rem}
We use the notion of oriented enclosed area of a curve in $S^2$ used in \cite{P}. It is only well-defined modulo $4 \pi.$
\end{Rem}
\begin{proof}
The proof is a straightforward application of the direct method of calculus of variations.
We want to find a minimizer of the Willmore energy in the set
$$\mathcal S := \{\gamma : S^1 \rightarrow S^2 \text{ smooth }| L(\gamma) = L_0 \text{ and } A(\gamma)= A_0\}.$$
By Theorem 1 of \cite{P} the set is non empty, if the isoperimetric inequality holds. Thus $\mathcal E_0:= $ inf$\{\mathcal E(\gamma) | \gamma \in \mathcal S\} \geq 0.$ 
Without loss of generality we only consider arclength parametrized curves. Let $(\gamma_n)_{n \in \N}$ be a sequence in $\mathcal S$ such that 
$$\lim_{n\rightarrow \infty}\mathcal E (\gamma_n) = \mathcal E_0.$$
Since we have
\begin{equation}
\begin{split}
\int |\gamma_n'|^2 ds &= L_0\\
\int |\gamma_n''|^2 ds &= \int (<\gamma_n'', N_n>^2 + <\gamma_n'', \gamma_n>^2)ds = \int (\kappa_n^2 + 1) ds,  
\end{split}
\end{equation}
the sequence $(\gamma_n)_{n \in \N} $ is bounded in $W^{2,2}$ and has a convergent subsequence in $W^{2,2}$ by the Arcela-Ascoli theorem. Let $\gamma_0  $ denote the limit of this subsequence, then $\gamma_0$ is at least $\mathcal C^1$. Therefore $(\gamma_n)_{n \in \N} $ and $(\gamma_n')_{n \in \N} $ converges point wise. Further, by the Gau\ss-Bonnet theorem the enclosed area can be computed as $A(\gamma_n) = 2 m \pi - \int_{\gamma} \kappa_n ds = A_0,$ where $m$ is the winding number of the curve. 
Thus $\gamma_0$ is a minimizer of $\mathcal  E$ for curves lying in 
$$\tilde{\mathcal S} := \{ \gamma : S^1 \rightarrow S^2, \gamma \in W^{2,2} | L(\gamma) = L_0 \text{ and } A(\gamma)= A_0\}.$$ 

\noindent
It remains to show that $\gamma_0$ is smooth. For this we rewrite the Euler-Lagrange equation. The Hopf fibration induces a $S^1-$fiberbundle with canonical connection on $S^3.$ A conformal parametrization of the Hopf torus $f_0$ corresponding to $\gamma_0$  is obtained by taking the horizontal lift $\tilde \gamma_0$ of $\gamma_0$ as the profile curve of $f_0$, see Proposition \ref{equivariant}. Note that the horizontal lift is well defined for $W^{2,2}$ curves and preserves the regularity.
 Let $(T, N, B)$ denotes the Fr\'enet frame of $\tilde \gamma_0$. Then $\gamma_0$ is a constrained elastic curve in $S^2$ if and only if there exist real constants $\lambda$ and $\mu$ such that the vector field 
 $$X = (\kappa^2 + \lambda)T + 2 \kappa' N + (2\kappa + \mu) B$$
 is parallel with respect to the Levi-Civita connection on $S^3.$
 Thus $\kappa$ is a BV function on a compact interval and therefore $\kappa\in L^{\infty}$. Thus one can use the Cald\'eron-Zygmund estimates and obtain smoothness for $\kappa.$  
 \end{proof}

\begin{Cor}
Every conformal class of the torus can be realized as a constrained Willmore immersion in the $3-$sphere. 
\end{Cor}
\begin{proof}
By \cite{P} the conformal type  of a Hopf torus is given by $(L/2, A/2)$ and the region, where the isoperimetric inequality holds covers the whole moduli space of conformal structures of tori. 
\end{proof}

\section{Constrained Elastic Curves in Space Forms}
Since the Willmore functional is M\"obius invariant, it seems to be more natural to consider a M\"obius invariant setup here.  Thus we consider $$\gamma: \R \rightarrow H^2, S^2, \R^2 \hookrightarrow \C P^1$$

\noindent
via affine coordinates. The M\"obius invariant of a map into $\C P^1$ is the Schwarzian derivative. It can be defined by the  following construction which can be found in \cite{BuPP}. Let $\gamma$ be a curve in $\C P^1.$ To $\gamma$ there exist a lift   $ \tilde \gamma$  to $\C^2$ (not necessarily closed) with respect to the canonical projection from $\C^2$ to $\C P^1.$ Further, there exists a complex valued function $a$ with 
$\hat \gamma := a \tilde \gamma$ such that $\det_{\C}(\hat \gamma, \hat \gamma') = 1.$
Thus $\hat \gamma''$ and $\hat \gamma$ are linearly dependent over $\C$ and there exists  a complex valued function $q$ satisfying 
\begin{equation}\label{HLL}
 \hat \gamma ''  + q \hat \gamma = 0.
 \end{equation}
\begin{Def}
The function $q$ is called the Schwarzian derivative of $\gamma.$
\end{Def}
\noindent 
The curve is uniquely determined by $q$ up to M\"obius transformations. A straightforward computation gives the following lemma. The lifts $\tilde \gamma$ needed for the computations are:
for $\R^2 \cong \C \hookrightarrow \C^2$ and $H^2 \hookrightarrow \R^2$ we use
$\tilde \gamma = (\gamma, 1)$,  and for $S^2$ we use $\tilde \gamma = \eta$, where $\eta \subset S^3 \subset \C^2$ is the horizontal lift of $\gamma$ under the Hopf fibration. 

\begin{Lem}\label{Lemma1}
Let $\gamma$ be a regular and arclength parametrized curve in a $2-$dimen-sional space form of constant curvature $G$ and let $\kappa$ be its geodesic curvature. 
Then the Schwarzian derivative $q$ of $\gamma$ is given by
$$q = \frac{i\kappa'}{2} + \frac{\kappa^2}{4}  + \frac{G}{4}.$$
Further, if $\gamma$ is constrained elastic in the space form, i.e., $\kappa$ is a real solution of the stationary mKdV equation (\ref{P4}) with real parameters $\lambda, \mu$ and $\nu,$ then $q$ satisfies the stationary KdV equation 
\begin{equation}\label{KdV}(q')^2 + 2 q^3 + c q^2 + 2d q + e = 0,
\end{equation}
with real parameters $c$, $d$ and $e$ given by
\begin{equation}
\begin{split}
c&= \mu - \tfrac{G}{2}\\
d&= -\tfrac{\nu}{4} - \tfrac{G^2}{16} - \mu \tfrac{G}{4}\\
e&= cd + \tfrac{\lambda^2}{4}+\tfrac{\mu^2 G}{4} - \tfrac{\nu G}{4}.
\end{split}
\end{equation}
\end{Lem}

\begin{Rem}The transformation $\kappa \mapsto q$ of  an arclength parametrized curve is a geometric version of the well-known Miura transformation, see for example \cite{GP}. 
\end{Rem}
\noindent
Let
\begin{equation*}
\begin{split}
g_2 &:= \frac{c^2}{12} - d = \frac{(\mu + G)^2}{12} + \frac{\nu}{4} \\
g_3 &:=  -\frac{cd}{12} +  \frac{e}{4} + \frac{1}{6^3}c^3 = \frac{1}{216}(\mu+ G)^3 + \frac{1}{16}\lambda^2 - \frac{1}{24}\nu (\mu + G)\\
P_3 &:= 4 x^3 -g_2 x - g_3.
\end{split}
\end{equation*} 
If $D= g_2^3 -27 g_3^2  \neq  0$ then the differential equation 
\begin{equation}\label{1}\wp '^2 = P_3(\wp)
\end{equation}
 defines a double periodic meromorphic function - the Weierstrass $\wp$ function. Its periods $\omega_i$ are linearly independent over the reals, i.e., the $\omega_i$ generates a lattice $\Gamma$ in $\C,$ and $\wp$ is a well-defined function on  $T^2 = \C/ \Gamma.$ The equation (\ref{KdV}) is then solved by 
$$q(x) =  - 2 \wp(x+x_0) - \tfrac{1}{6}c,$$
for some constant $x_0 \in \C\setminus{\{0\}}.$  We refer to \cite{AS} for details on the Weierstrass elliptic functions.\\

A necessary condition for $q$ to be the Miura transformation of a real valued curvature function $\kappa$ is that the lattice invariants  $g_2$ and $g_3$ are real. We also need that  $D\neq 0$  to obtain a well-defined $\wp-$function. This requires the polynomial $P_3$ to have only simple roots. Then the generators of the lattice $\Gamma$ are linearly independent over the reals. We deal with the case of $P_3$ having multiple roots in section \ref{multiple roots}. For real $g_2$ and $g_3$, the lattice $\Gamma$ is rectangular or its double covering is rectangular, depending on the sign of its discriminant. 

\begin{Def}
A solution of equation (\ref{KdV}) with $D > 0$ is called 
orbitlike and wavelike, if $D < 0.$
The polynomial $P_3$ has multiple roots if and only if $D = 0.$ 
\end{Def}
\begin{Rem}
For given parameters $\mu$ and $\lambda$ consider the trajectories of solutions to equation (\ref{secondorder}) with different initial values. The trajectories of the constant solutions mark special points in the $(\kappa, \kappa')-$plane.  If the equation (\ref{secondorder}) possesses orbitlike solutions, then there exist three constant solutions and the trajectories of orbitlike solutions only wind around one  of these constant solutions, i.e., they lie in the orbit of the constant solution. Wavelike solutions always wind around all constant solutions of the equation. For $\lambda = 0$ the periodic solutions $\kappa$ changes the sign, thus the corresponding curves resemble waves.
\end{Rem}
A curve in $\C P^1$ with Schwarzian derivative $q$  solving equation (\ref{KdV}) can be parametrized in terms of Weierstrass $\zeta$ and $\sigma$ functions. 
The Weierstrass $\zeta-$function is determined by $\zeta' = - \wp$ and $\lim_{z \to 0}( \zeta(z) - \tfrac{1}{z}) = 0$ and the Weierstrass $\sigma$ function is given by $\tfrac{\sigma'}{\sigma} = \zeta$ and $\lim_{z \to \infty} \tfrac{\sigma(z)}{z} = 1.$
Again, we refer to \cite{AS} for the properties of these functions.
\begin{The}\label{curves}
Let $\tilde q = -2 \wp(x+ x_0) - \tfrac{1}{6}c$  be a solution of equation (\ref{KdV}) with real parameters $c, d, e$. We define a family of curves $\hat \gamma_E = (\hat\gamma^1_E, \hat\gamma^2_E)  \subset \C^2, $ $E \in \R$ by 
\begin{equation}
\begin{split}
\hat\gamma^1_E &=  \frac{\sigma(x+x_0 - \rho)}{\sigma(x+x_0)}e^{\zeta(\rho)(x+x_0)}\\
\hat\gamma^2_E &= \frac{\sigma(x+x_0 + \rho)}{\sigma(x+x_0)}e^{\zeta(-\rho)(x+x_0)}, \quad \text{with } \wp(\rho) = E.
\end{split}
\end{equation}
Then $\hat \gamma_E$
induces a family of curves $\gamma_E$ in $\C P ^1$ with Schwarzian derivative $q_E = (\tilde q + \tfrac{1}{6}c - E),$ if $E$ is not a branch point of $\wp.$ \end{The}
\begin{Rem}
The parameter $\rho$ is determined by $E$ only up to sign. The choice of $-\rho$ (instead of $\rho$) exchanges $\gamma^1_E$ and $\gamma^2_E$ and the resulting curves in $\C P^1$ are M\"obius equivalent.
\end{Rem}
\begin{proof}
 If $E$ is not a branch point of $\wp,$ the functions $\hat\gamma^i_E,$  $i= 1,2$ are linearly independent over $\C$ and have no common poles and zeros, thus the curve $(\gamma^1_E, \gamma^2_E)$ induces a well-defined curve in $\C P^1.$  
Further, since
$$(\hat\gamma_E^i)'' - 2 \wp(x + x_0) \hat\gamma_E^i = E \hat\gamma_E^i,$$
the stated $q_E$ is  the Schwarzian derivative of the curve $\gamma_E = [\gamma^1_E, \gamma^2_E].$
\end{proof}
\begin{Lem}\label{compa}
Let $g_2$ and $g_3$ be real constants with $g_2^3 - 27g_3^2 \neq 0.$ And let $\wp$ be the Weierstrass function with respect to the lattice $\Gamma \subset \C$ given by the lattice  invariants $g_2$ and $g_3$. If $x_0\in \C \setminus (\tfrac{1}{2}\Gamma + \R)$, then there exist a function $\kappa: \R \rightarrow \R$ with 
\begin{equation}\label{statement}\wp(x + x_0)= - i \frac{\kappa'(x)}{4} -  \frac{\kappa(x)^2}{8} - b,\end{equation}
where $b$ is a real constant.
Moreover, $\kappa$ is periodic and a stationary mKdV solution with coefficients determined by $g_2,$ $g_3$.
\end{Lem}
\begin{proof}
We first show that there exists a real valued function $\kappa$ solving the differential equation (\ref{statement}). The imaginary part of (\ref{statement}) can be easily integrated and we obtain
\begin{equation}\label{kappazeta}
\kappa := -2 i ( \zeta - \bar \zeta + const_1).
\end{equation}

\noindent
Then the real part of equation (\ref{statement}) must satisfy
$$ \wp + \bar \wp =  (  \zeta - \bar \zeta + const_1)^2 - 2b,$$
which can be proved as follows: Differentiating equation (\ref{1}) we obtain
\begin{equation}\label{2}
\wp''(x+ x_0) = 6 \wp(x+x_0)^2 - \tfrac{1}{2}g_2.
\end{equation}

\noindent
Further, since the functions $\wp$ and $\bar \wp$ are holomorphic and anti-holomorphic, respectively, we get that  the  derivative  of  $\wp$ with respect to $z = x+iy$ and the derivative of $\bar \wp$ with respect to $\bar z$ is the same as the derivative of $\wp$ and $\bar \wp$ with respect to $x.$  Consider now only the points $z \in \C /\Gamma $ with $\wp-\bar \wp \neq 0.$ Then by (\ref{1}) and (\ref{2}) we have
 \begin{equation*}
2 (\bar \wp - \wp)^3 = (\wp'' + \bar \wp'')(\bar \wp - \wp) + (\wp')^2 - (\bar \wp')^2 .
\end{equation*}
This is equivalent to 
\begin{equation*} 
2 (\bar \wp - \wp) = \frac{\wp'' + \bar \wp''}{\bar \wp - \wp} + \frac{(\wp')^2 - (\bar \wp')^2}{(\bar \wp - \wp)^2}.
\end{equation*}
By integration we get
\begin{equation*}
 2(\zeta - \bar \zeta + const_1) = \frac{\wp' + \bar \wp'}{\bar \wp - \wp}, 
 \end{equation*}
 with a purely imaginary integration constant $const_1.$ Thus
 \begin{equation*}
\wp' + \bar \wp' = 2 (\bar \wp - \wp)(\zeta - \bar \zeta + const_1).
\end{equation*}
Integrate again we obtain
\begin{equation*}
\begin{split}
\wp + \bar \wp& =  (( \zeta - \bar \zeta) + const_1)^2 + const_2,
\end{split}
\end{equation*}
with a real integration constant $const_2.$ Then replacing $\wp$ by $\wp(x + x_0)$ and define $b = - \tfrac{1}{2}const_2$ proves the first statement.\\

Since all the functions we consider are continuous the equation above is still valid at the boundaries in the $x-$direction. 
 Thus it is necessary to choose a $x_0$ which does not lie on the real axis or on a parallel translate of the real axis by a half lattice point. These choices of $x_0$ does not lead to an arclength parametrized constrained elastic curve, since $q$ would be real valued. \\

Now we show that $\kappa$ defined by equation (\ref{kappazeta}) is mKdV stationary. We have
$\wp = - i \frac{\kappa'(x)}{4} -  \frac{\kappa(x)^2}{8} - b$  and therefore
\begin{equation*}
\begin{split}
\wp(x+x_0)'' &=  - i \frac{1}{4} \kappa'''(x)-  \frac{1}{4}\kappa''(x)\kappa(x) -  \frac{1}{4} (\kappa'(x))^2\\
6 \wp(x+x_0)^2 &=  \frac{3i}{8} \kappa' \kappa^2 +  3ib\kappa'  - \frac{3}{8}\kappa'^2+  \frac{3}{32}\kappa^4 +  6b^2 +  \frac{3}{2}b\kappa^2.
\end{split}
\end{equation*}
Hence the imaginary part of equation (\ref{2}) yields
\begin{equation}
\begin{split}
  \kappa''' + \frac{3}{2}\kappa'\kappa^2 + 12b\kappa' = 0.\end{split}
\end{equation}
Thus $\kappa$ is the curvature of a arclength parametrized constrained elastic curve. 
\end{proof}
\begin{Rem}
Lemma \ref{compa} shows that the curve $\gamma_E$ with Schwarzian derivative $q$ defined in Theorem \ref{curves} is M\"obius equivalent to an arc length parametrized constrained elastic curve $\gamma$ in a $2-$dimensional space form. We fix the M\"obius transformation in section \ref{spaceform}.
\end{Rem}

\subsection{The roots of the polynomials $P_3$ and $P_4$}
Since we want to consider closed curves, the curvature function $\kappa$ is periodic and achieves its maximum and minimum. Thus we can always  choose $\kappa'(0) = 0$ as the initial value for the equation (\ref{secondorder}). 
This corresponds to the choice of $x_0 \in i \R\setminus \{ \tfrac{1}{2} \Gamma\}.$ 
The necessary and sufficient condition  for the existence of a real function $\kappa$ solving equation (\ref{P4}) with parameters $\mu,$ $\lambda$ and $\nu$ is that the polynomial $P_4$ has real roots. 
In the case of a $4-$th order polynomial there exists an algorithm to compute its roots explicitly. To $P_4$ one associate a polynomial of degree $3$ - the cubic resolvent. In our case it is  given by 
$$\tilde P_3 = s^3 + 8(\mu + G)s^2 + 16((\mu + G)^2 -\nu)s - 64 \lambda^2$$
By a variable change $16x  =  s + \tfrac{8}{3}(\mu + G) $, we obtain a positive multiple of the  polynomial $P_3.$ The roots of $P_4$ are determined by the roots of $\tilde P_3$ (respectively $P_3$). In particular, $P_4$ has simple real roots if and only if  $\tilde P_3$ has either only one real root ($D<0$ and $P_4$ has $2$ real roots) or the roots of $\tilde P_3$ are all real and non-negative ($D>0$ and $P_4$ has $4$ real roots). Further, if $\tilde P_3$ has multiple roots, then also $P_4$ has multiple roots. This yields the following lemma.

 \begin{Lem}\label{positive}
Let $P_4$ be the real polynomial of degree $4$ given in (\ref{P4}) with only simple roots  and let $\tilde P_3$ denote its cubic resolvent. Then $P_4$ has real roots if and only if   all real roots of $\tilde P_3$ are  non-negative.
\end{Lem}
\begin{proof} 
The statement is obviously true for $D>0.$ For $D<0$ let $e_1$, $e_2$ and $e_3$ denote the roots of $\tilde P_3$.
Then the cubic resolvent can be written as $\tilde P_3(s) = (s-e_1)(s-e_2)(s-e_3)$.
We obtain in our particular case that
$$\tilde P_3(0)  = - e_1 e_2 e_3 = - 64 \lambda^2 \leq 0.$$ 
For $D < 0$ there is only $1$ real root and a pair of complex conjugate roots of $P_3.$ Therefore the real root must be non-negative.
\end{proof}
\begin{Rem}
The proof shows that for given $g_2$ and $g_3$ and $(\mu + G)$ the parameter $\lambda$ is fixed up to sign. The choice of the sign corresponds to the transformation $\kappa \mapsto -\kappa$ or equivalently $x_0 (\in i\R) \mapsto -x_0.$
\end{Rem}

\begin{Cor}\label{g3}
The stationary mKdV equation (\ref{P4}) with real parameters $(\mu + G),$ $\lambda$ and $\nu$ has real solutions if and only if $\tfrac{1}{6} (\mu + G)$ is  less or equal to  all real roots of the polynomial $P_3.$ 
Equality holds if and only if $\lambda = 0.$
\end{Cor}
\begin{Cor}
There exist no orbitlike free elastic curves on $S^2$. Further, there are no orbitlike elastic curves corresponding to Willmore Hopf tori.\end{Cor}
\begin{proof} 
Firstly, it requires $g_2 > 0$ to have $D>0.$ Further, the condition for the existence of real solutions is equivalent to the condition that the roots of $\frac{\partial \tilde P_3}{\partial s} = 3 s^2 + 16(\mu + G)s + 16 (\mu + G)^2 -16 \nu $ are positive\footnote{If the maximum and the minimum of the polynomial are positive, then at least $2$ roots must be positive. But Since the product of all roots is also non-negative by lemma \ref{positive}, the third root is non-negative.}. This condition is computed to be
$$-\sqrt{\frac{64}{3}g_2} \geq \tfrac{8}{3} (\mu + G),$$
which is equivalent to 
$$(\mu + G) < 0 \text{ and } \nu \leq (\mu + G)^2.$$
But for free elastic curves in $S^2$ we have: $G > 0,$ and $\lambda = \mu = 0$ and
for Willmore Hopf tori we have : $G > 0$, $\lambda = 0$ and  $(\mu + G) = \tfrac{1}{2}G > 0.$
\end{proof}

\subsection{Multiple roots}\label{multiple roots}
We have shown that in the case where the polynomial $P_4$ has only simple roots the equation (\ref{KdV}) can be solved using the Weierstrass $\wp-$function. Now we study the case where $P_4$ has multiple roots. \\

Since we are looking  for periodic solutions, we can restrict ourselves without  loss of generality  to the initial value problem for equation (\ref{secondorder}) with initial values 
$$\kappa(0) = \kappa_0 \quad \text{ and } \quad \kappa'(0) = 0.$$
Then $\kappa_0$ is a real root of $P_4$ with parameters $\lambda,$ $\mu$ and $\nu.$ There are two cases to consider.
In the first case $\kappa_0$ is a multiple zero of $P_4$ itself. Then it is also a root of $\frac{\del P_4}{\del \kappa}$, which is the right hand side of equation (\ref{secondorder}). Therefore $\kappa \equiv  \kappa_0$ is the unique solution to the given initial value problem by Picard-Lindel\"off. \\
In the second case $P_4$ has multiple roots but $\kappa_0$ is a simple root of $P_4$.

\begin{Def}
A solution of equation (\ref{secondorder}) (or of equation (\ref{KdV})), where $P_4$ has multiple roots and the initial condition $\kappa_0$ is a simple root is called an asymptotic solution.
\end{Def}
\begin{Pro}
Asymptotic solutions with $\lambda=0$ are never periodic.  
\end{Pro}
\begin{proof}
For $\lambda=0$ we have the differential equation
$$(\kappa')^2  = - \tfrac{1}{4}\kappa^4 -  2(\mu + G)\kappa^2 - \nu.$$
The polynomial on the right hand side is even and has multiple roots by assumption. In order to obtain non constant solutions we need at least $1$ simple root of $P_4.$ 
By symmetry the only case to consider is that the multiple root of $P_4$ is at 
$\kappa = 0$ with multiplicity $2$ and we have $2$ simple roots for $\kappa = \pm \kappa_0,$ and $\kappa_0 \in \R_+.$\\

We solve an initial value problem for the differential equation of second order
$$\kappa'' + \tfrac{1}{2}\kappa^3 + (\mu + G)\kappa = 0,$$
with initial value $\kappa(0)= \kappa_0$ and $\kappa'(0) = 0.$
At $\kappa(0)$ we obtain that $\kappa''(0) = \frac{\del (\kappa')^2}{\del \kappa}|_{x = 0} < 0$. Thus there exist an $\epsilon >0$ with $\kappa'(t) < 0 $ for $t \in (0, \epsilon)$ and  the curvature function $\kappa$ decreases monotonically for $t \in (0, \epsilon).$ Let $T: =$ sup$\{\epsilon \in \R_+ | \kappa'(t) < 0$ for $t \in (0, \epsilon)\}$. If $T< \infty,$ then $\kappa'(T)= 0$  and we obtain $\kappa(T)$ is a root of $P_4$. Since $\kappa$ is continous, we obtain
$\kappa(T)=0,$ which is a multiple root. By Picard-Linderl\"off we get that $\kappa(t) \equiv 0$ is the unique solution to the initial value problem $\kappa'(T) = \kappa(T) = 0.$ This contradicts   $\kappa(0)= \kappa_0 \neq 0.$ Therefore  $T = \infty$ and $\kappa$ is not periodic.
\end{proof}
\begin{Cor}
Constrained Willmore tori of revolution and Willmore Hopf tori  are either homogenous, i.e., $\kappa \equiv \kappa_0$ is constant,  or $P_4$ has only simple roots. 
\end{Cor}

\begin{Rem}
Closed asymptotic solutions corresponding to constrained Willmore tori do exist for curves in $S^2$. These are obtained by a simple factor dressing of a multi-covered circle. 
In fact all asymptotic solutions on $S^2$ can be obtained this way. 
\end{Rem}

\subsection{Closing Conditions}\label{CC}
To obtain closing conditions for the curves $\gamma_E$ defined in Theorem \ref{curves} we compute their monodromy. The curve $\gamma_E$ closes if and only if the monodromy is a rotation by a rational angle.
We fix a lattice $\Gamma$ in $\C$ with real lattice invariants $g_2$ and $g_3$ and get  a $\wp-$function with respect to this lattice. We denote by $\omega_i,$ $i= 1,2,3,$ the  half periods of $\Gamma$ and fix $\omega_1$ to be the half period lying on the real axis. For real $g_2$ and $g_3$  we always obtain a half lattice point on the imaginary axis, which we denote by $\omega_3.$ In the case of $D < 0$ we have $\omega_1 = \omega_3$ mod $\Gamma.$

\begin{Pro}
With the notations above the curve $\gamma_E$ closes after $n$ periods of the Weierstrass $\wp-$function if and only if there exist a $m \in \N$ with $gcd(m,n) = 1$ such that 
$$ 2 \eta_1 \rho - 2\zeta(\rho) \omega_1  = \frac{m}{ n}\pi i.$$ 
Here $\zeta$ is the Weierstrass $\zeta-$function, $\eta_1 := \zeta(\omega_1)$ and $E  = \wp(\rho).$ 
\end{Pro}
\begin{Rem}
Geometrically speaking, the number $m$ is the winding number of the curve and the number $n$ the lobe number.
\end{Rem}
\begin{proof}
Provided that $E$ is not a branch point of the $\wp$-function the curve $\gamma_E = [\hat\gamma_E^1, \hat\gamma_E^2]$ is given by two complex valued functions 
\begin{equation*}
\begin{split}
\hat\gamma^1_E &=  \frac{\sigma(x+x_0 - \rho)}{\sigma(x+x_0)}e^{\zeta(\rho)(x+x_0)}\\
\hat\gamma^2_E &= \frac{\sigma(x+x_0 + \rho)}{\sigma(x+x_0)}e^{\zeta(-\rho)(x+x_0)}, \quad \text{with } \wp(\rho) = E.
\end{split}
\end{equation*}
Further, let $\zeta$ be the Weierstrass $\zeta$-function and  define $\eta_1 := \zeta(\omega_1)$, which is a real number because the lattice invariants $g_2$ and $g_3$ are real.
With the formulas for the monodromy of the Weierstrass $\sigma$ function we obtain:
\begin{equation*}
\begin{split}
\hat \gamma_E^1(x+ 2 \omega_1) &= e^{-2 \eta_1 \rho+ 2\zeta(\rho)\omega_1}  \hat \gamma_E^1(x)\\
\hat \gamma_E^2(x+ 2 \omega_1) &=e^{2 \eta_1 \rho- 2\zeta(\rho)\omega_1}  \hat \gamma_E^2(x).
\end{split}
\end{equation*}
The monodromy of the $\gamma_E$ is the quotient of the both monodromies computed here. 
Therefore we get that the curve closes after $n$ periods if and only if there exist a $m\in \Z$ with $(m,n)$ coprime  such that 
\begin{equation*}
 e^{ 4\eta_1 \rho- 4\zeta(\rho)\omega_1} = e^{\frac{m}{n} 2\pi i},
 \end{equation*}
 which proves the statement.
 \end{proof}
\begin{Cor}
Varying $x_0$ yields {\em isospectral} deformations of constrained elastic curves, i.e., deformations preserving the monodromy and the parameters $g_2$, $g_3$ and $E$.  In particular, every constrained elastic curve is isospectral to an elastic curve,  i.e.,  a solution of equation (\ref{secondorder}) with $\lambda = 0,$ unique up to reparametrization.
\end{Cor}
\begin{proof}
Varying $x_0$ does not effect the closing condition, thus we obtain a $1-$parameter family of closed constrained elastic curves. 
For the second statement we define  $\tfrac{1}{6}(\mu + G):= \wp(\omega_3),$ which is by definition the smallest real root of $P_3.$ Thus we have $\lambda = 0$ and $\nu = 4g_2 - 12 \wp(\omega_3)^2.$ This choice of parameters leads to an elastic curve since the so defined $P_4$ has real roots by lemma \ref{positive}. 
The corresponding $x_0$ can be determined as follows: 
The roots of $P_4$ are given by 
\begin{equation}
\begin{split}
\kappa_0^1 &= \sqrt{-24 \wp(\omega_3) + \sqrt{624 \wp^2(\omega_3) - 16 g_2}}\\
  \kappa_0^2 &= - \sqrt{-24 \wp(\omega_3) + \sqrt{624 \wp^2(\omega_3) - 16 g_2}}\\
    \kappa_0^3 &= \sqrt{-24 \wp(\omega_3) - \sqrt{624 \wp^2(\omega_3) - 16 g_2}}\\
      \kappa_0^4 &= -\sqrt{-24 \wp(\omega_3) - \sqrt{624 \wp^2(\omega_3) - 16 g_2}},
      \end{split}
      \end{equation}
if the solution is orbitlike. For wavelike solutions there are only $2$ real roots which are given by $\kappa_0^1$ and $\kappa_0^2.$\\

\noindent
 Thus the possible values of $\wp(x_0)$ are
$$\wp(x_0) = - \tfrac{7}{2}\wp(\omega_3) - \tfrac{1}{8}\sqrt{624 \wp^2(\omega_3) - 16 g_2},$$$$ Ê\wp(x_0) = - \tfrac{7}{2}\wp(\omega_3) + \tfrac{1}{8}\sqrt{624 \wp^2(\omega_3) + 16 g_2}.$$
The first choice corresponds to $\pm x_0 \in i\R$ and the second to $\pm x_0 i\R + \omega_1$. Both choices yield the same curve up to reparametrization and there exist a unique $x_0 \in  i (0,-i \omega_3)$  such that  $\kappa_0$ is a root of $P_4$. The choice of $x_0 \in (i \omega_3, 0))$ leads to the same curves with different orientation, since the map $x_0 \mapsto -x_0$ corresponds to $\kappa \mapsto - \kappa$. 
\end{proof}
\noindent
Because of the above corollary, we restrict ourselves in the following to the case with $\lambda = 0.$

\begin{The}\label{Hopf1}
Let $g_2$ and $g_3$ be real constants with $g_2^3 - 27g_3^2  \neq 0.$ Then every rational point of the function
$$g : i\R \setminus \{\omega_3 \Z\} \rightarrow i \R, \rho \mapsto g(\rho) =  \eta_1 \rho - \zeta(\rho) \omega_1$$
gives rise to a closed elastic curve $\gamma_E,$ $E = \wp(\rho),$ as defined in Theorem  \ref{curves}, on a round $S^2$ with curvature $G = 4 (\wp(\omega_3) - E)  $. In particular, for fixed $g_2$ and $g_3$ there exist to every integer $n$  a  simply closed  elastic curve with $n$ lobes.
\end{The}

\begin{proof}
The polynomial $P_3$ defining the Weierstrass $\wp-$function has either $1$ or $3$ real roots. By assumption $\wp(\omega_3) = \tfrac{1}{6}(\mu + G),$ where $\omega_3 \in i\R$ is a half lattice point of $\Gamma.$  We vary $\rho,$ with $\wp(\rho) = E,$ to close the curves. Since $E = \tfrac{1}{6}(\mu - \tfrac{1}{2}G) < \wp(\omega_3),$
we obtain $\rho \in i\R\setminus \{\omega_3 \Z\}$, see \cite{AS}.
For fixed real invariants $g_2$ and $g_3$ we get that $\eta_1$ and $\omega_1$ are also real. Further, for $\rho \in i \R$ the constant  $\zeta(\rho) \in i\R,$ too.
Thus the map
 $$g : i\R \rightarrow i \R, g(\rho) =  \eta_1 \rho - \zeta(\rho) \omega_1,$$
is well defined and  $g(i \R)$ is a nontrivial interval since
 $$\lim_{\rho \rightarrow \pm 0} g(\rho )  = \pm \infty \text{ and } g(\omega_3) = 0 \text { or } g(\omega_3) = \tfrac{1}{2} \pi i,$$
 depending on whether the solution is orbitlike or wavelike. 
\end{proof}
\begin{Rem}
For constrained elastic curves in $S^2$ it is necessary to choose $\rho \in i \R$ mod $\Gamma.$ Thus it is isospectral to an elastic curve in  a space form of positive curvature and $\tfrac{1}{6}(\mu + G) > E = \tfrac{1}{6}(\mu - \tfrac{1}{2}G).$ Nevertheless, by decreasing $\tfrac{1}{6}(\mu + G)$ for fixed $g_2$, $g_3$ and $E$\footnote{By choosing $\tfrac{1}{6}(\mu + G)$ according to Lemma \ref{g3} the parameter $\lambda$ is determined up to sign and $\nu$ is fixed and there is a  $x_0 \in i (0, -i \omega_3)$ with $\wp(x_0)= - \tfrac{ \kappa_0^2}{8} - \tfrac{1}{12}(\mu + G)$. Therefore varying $\tfrac{1}{6}(\mu + G)$ is equivalent to the isospectral deformations given by varying $x_0.$},   the resulting curves first become a constrained elastic curve in $\R^2$ for $\tfrac{1}{6}(\mu + G) = E$ and then turns into a constrained elastic (but not elastic) curve in $H^2.$
\end{Rem}
\begin{figure}[htbp]
  \centering
  \begin{minipage}[b]{5.5 cm}
    {\vbox{\hspace{-0.4cm}
\vbox{\vspace{0.5cm}
\includegraphics[scale=0.125]{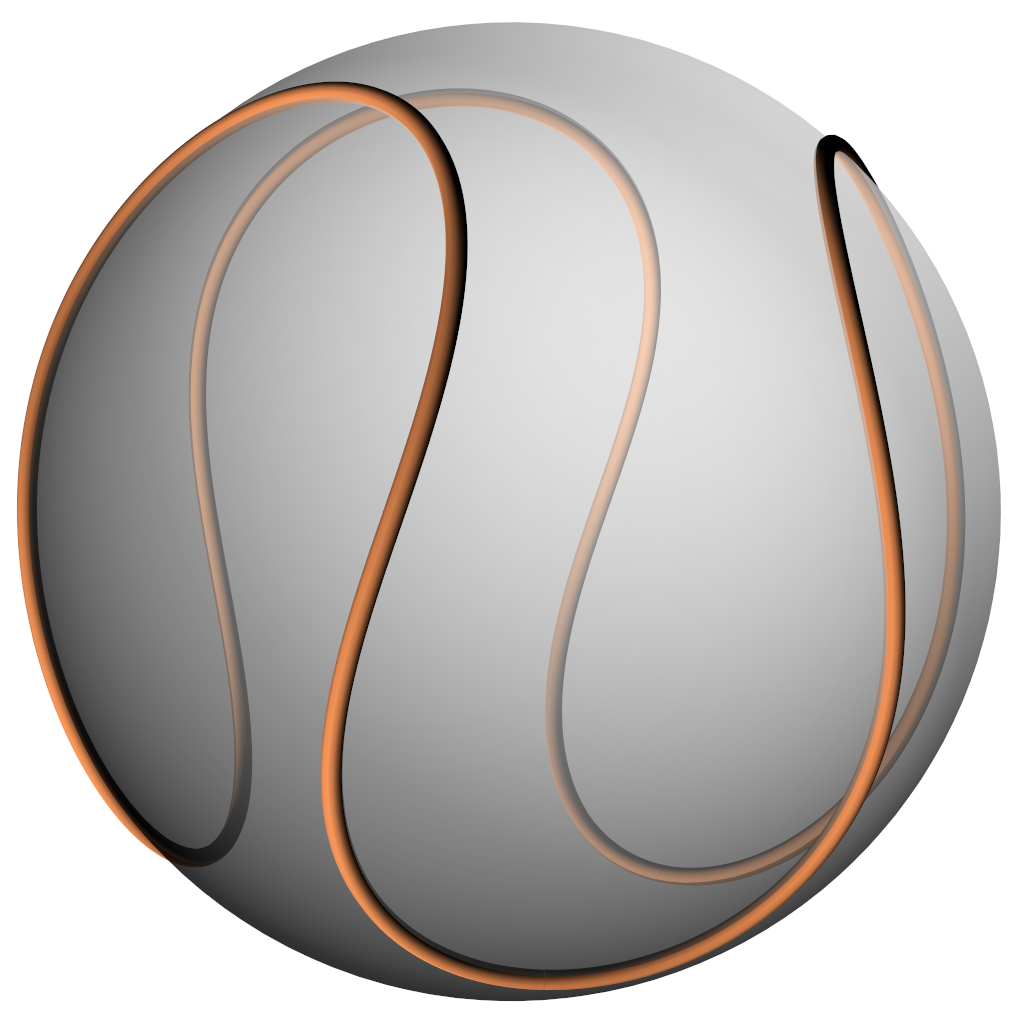}
}
}}  
  \end{minipage}
  \begin{minipage}[b]{5.5 cm}
   {\vbox{\hspace{0.5cm}
\vbox{\vspace{0.5cm}
\includegraphics[scale=0.13]{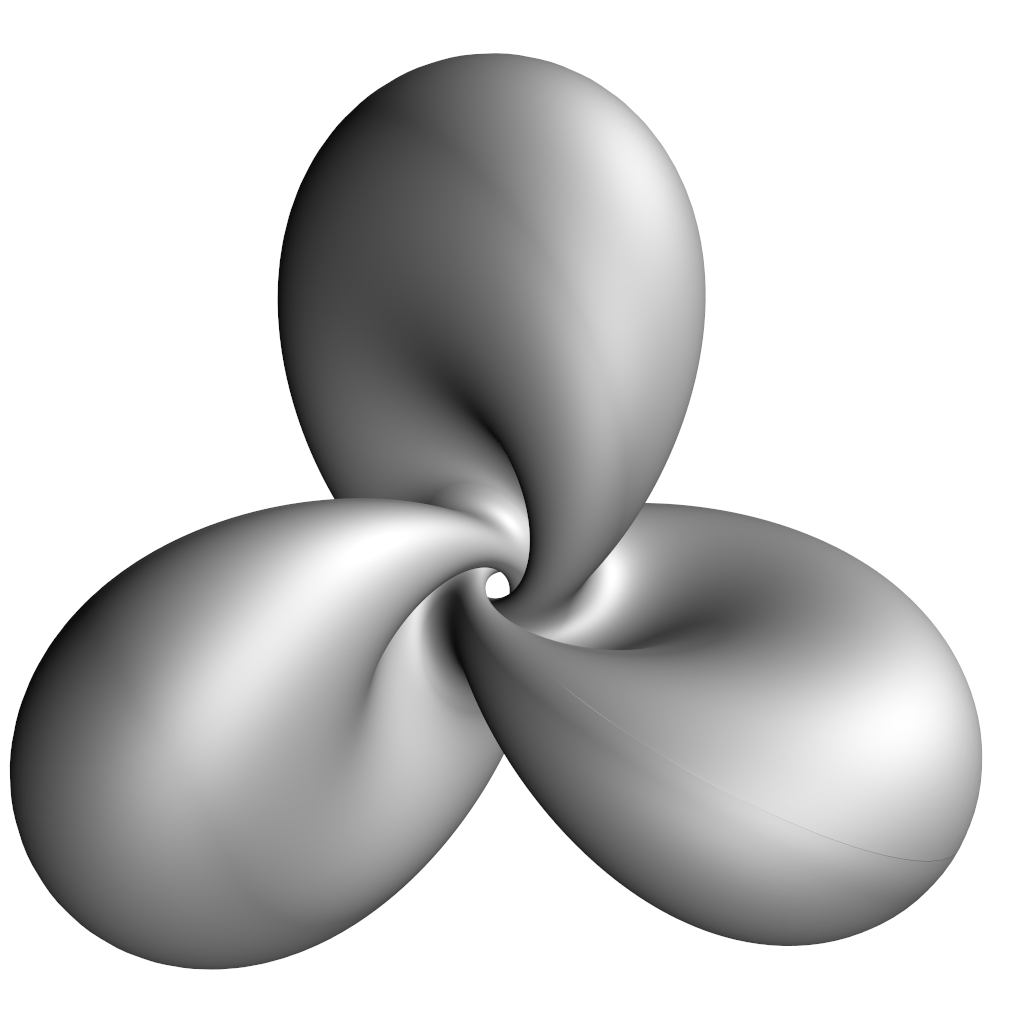}
}
}}
  \end{minipage}
  \caption{Wavelike elastic curve in $S^2$ to parameters $\mu = -\tfrac{1}{2}$ and $\lambda = 0$ in $S^2$ and corresponding Willmore Hopf torus.(by Nick Schmitt)}
  \label{Labelname}
\end{figure}

\begin{Pro}
Let $g_2$ and $g_3$ be real constants with $g_2^3 - 27g_3^2 < 0$ and $\gamma_E$ be the family of curves defined in Theorem \ref{curves}. Then there exists at most one closed elastic curve in a space form of constant curvature $G < 0 $ in that family.  \\
\end{Pro}
\begin{proof}
In this case $e = \tfrac{1}{6}(\mu + G)$ is the only real root of $P_3.$ Further $\rho$ with $\wp (\rho) = E > e$ does not lie on the imaginary axis.
Since $E$ must be real, we get $\rho \in \R$ and thus $\zeta(\rho) \in \R$.
Therefore the only chance to get a closed solution is that
$$\rho \eta_1 - \zeta(\rho) \omega_1 = 0.$$
The solution holds obviously for $\rho = \omega_1$ but this choice contradicts the fact that $E > e$. The closing condition can be interpreted as the intersection of the line given by $\rho \mapsto \rho \tfrac{\eta_1}{\omega_1}$ with the graph of the function $\zeta|_\R.$ The function $\zeta|_\R$ is anti-symmetric with respect to $\omega_1$ and has a simple pole in  $0$ and is convex for $\rho < \omega_1$ and concave for $\rho > \omega_1.$ Thus there exist two other intersection points if  and only if $-\wp(\omega_1) = - (E + \tfrac{1}{4}G) > \tfrac{\eta_1}{\omega_1},$ which makes the same curve. Otherwise there are no other intersection points and no closed curves.
\end{proof}
\begin{Exa}
A closed curve in this class is a elastic figure-eight in $H^2$. It is shown in \cite{LS} that there is no free elastic wavelike curve in the hyperbolic plane. Thus there is no Willmore torus coming from this construction.
\end{Exa}

\begin{The}
Let $g_2$ and $g_3$ be real constants with $g_2^3 - 27g_3^2 > 0.$ 
Then every rational point of the function
$$g :  i\R \setminus \{\omega_3 \Z\} \rightarrow i \R, g(\tilde \rho) =  \eta_1 (\tilde\rho + \omega_1 ) - \zeta(\tilde\rho + \omega_1 ) \omega_1$$
gives rise to a closed constrained elastic curve $\gamma_E$ ($E = \wp(\rho)$) as defined in Theorem  \ref{curves}) in  $H^2$ with curvature $G$. In particular, for fixed $g_2$ and $g_3$ there exist to every integer $n >1$  a  simply closed elastic curve with $n$ lobes.
\end{The}
\begin{proof}
The polynomial $P_3$ has three real roots and thus we can choose a $E > \tfrac{1}{6}(\mu + G)$ such that $P_3(E) < 0$ by varying $G < 0.$ The corresponding $\rho$ satisfies $ \rho = \tilde \rho + \omega_1$ with $\tilde \rho \in i \R$ and $$\overline{\zeta(\tilde \rho + \omega_1)} = - \zeta(\tilde \rho - \omega_1) = - \zeta(\tilde \rho + \omega_1) +  2 \eta_1.$$ Thus the function
$$g(\tilde \rho) =  \eta_1 (\tilde\rho + \omega_1 ) - \zeta(\tilde\rho + \omega_1 ) \omega_1$$	
is purely imaginary.  Further  $g(\omega_3) = \tfrac{1}{2}\pi i$ and $g(0) = 0$.  By the same argument as in Theorem \ref{Hopf1} we get a dense set of solutions. In particular, for $n>1$ we obtain $\tfrac{1}{2n} \pi i \in g(i \R).$  
\end{proof}
\begin{Rem}
In contrast to constrained elastic curves in $S^2$, elastic curves in $H^2$ never lie in a isospectral family of constrained elastic curves in other space forms.
\end{Rem}

\subsection{How to obtain the Space form}\label{spaceform}
We want to  show that the curves stated in Theorem \ref{curves} are already the constrained elastic curves we are looking for without applying any M\"obius transformations. We use the Poincare disc model or the upper half plane model of $H^2 \hookrightarrow \C$ (depending whether the function $g$ defined above is real or imaginary valued) and consider $S^2 = \C \cup \{\infty\}.$   The curve $\gamma_E$ given by Theorem \ref{curves}  is M\"obius equivalent to a constrained elastic curve $\gamma$ in a space form $\mathcal G$ of constant curvature $G$. A M\"obius transformation $M$ is fixed by its values on $3$ points.  We want to determine the M\"obius transformation $M$ from $\mathcal G$ to $\C P^1$ which maps the arclength parametrized constrained elastic curve $\gamma$ to $\gamma_E.$  Without loss of generality we can fix $\gamma(0) \in i\R.$\\

For real valued parameter $E$ the function $g$ is either real or imaginary valued.  In the first case the monodromy is a rotation which has two fixed points $0$ and $\infty$ in $\C P^1$ and this rotation must be an isometry of $\mathcal G.$ Thus we use the poincare disc model of the hyperbolic plane here. Since the inversion at the unit circle preserves the  constrained elastic property of a curve, the only M\"obius transformations left are $z \mapsto r z,$ for a real number $r.$ We can fix $r$ by asking the curve $\gamma_E$ to be arc length parametrized with respect to the induced metric (which we need only to check in one point), i.e., $|\gamma_E'(0)|_{\mathcal G}^2= 1$. In the second case (which only happens for constrained elastic curves in $H^2$), the hyperbolic space is given by the upper half plane and again the arclength property fixes the parameter $r.$ The choice of $r$ corresponds to the choice of the infinity boundary of the hyperbolic plane or respectively the image of the geodesic under the stereographic projection of $S^2.$ If the space form is $\R^2$, then multiplication with $r$ preserves the constrained elastic property.\\

\subsection{Constrained Willmore cylinders of revolution}
Constrained Willmore cylinders of revolution have constant mean curvature (CMC) in a $3-$dimensional space form by \cite{B}. For tori we have two cases to distinguish. Either the whole torus is CMC in one space form or the torus is constructed by the glueing of two CMC cylinders in the hyperbolic $3-$space (viewed as the inner of the unit ball in $\R^3$ for one cylinder and as the outer of the unit ball for the other cylinder) at the infinity boundary. In both cases we can associate to the immersion a Riemann surface - the spectral curve. The details concerning the construction of the immersed surfaces and its corresponding spectral curves can be found in \cite{Bob} in the first case in \cite{BaBob} in the second.   For a constrained Willmore torus of revolution its CMC spectral curve is  determined by the family of differential operators 
\begin{equation*}
D_1^a = \del_x
+    \begin{pmatrix}- ia &i \frac{\kappa}{2} \\ i \frac{\kappa}{2} & ia
         \end{pmatrix},
\end{equation*}
see \cite{H1},
where $\kappa$ is the curvature of its profile curve in the hyperbolic plane (G= -1) and $a \in \C \setminus\{0\}$. To be more concrete, the spectral curve is given  by the normalization and compactification of the analytic variety
$$\{(a, b) \in \C\setminus\{0 \} \times  \C\setminus\{0 \}| \text{b is  eigenvalue of the holonomy of } D_1^a\}.$$
The so defined spectral curve is a hyperelliptic curve over the $a-$plane and there exist by construction two involutions which cover the involutions 
$$\sigma: a \mapsto -a \text{ and } \rho: a \mapsto \bar a$$
on the $a-$plane. The spectral curve  is unbranched over $a \in \R$ (since the corresponding $D^a_1$ are in $\mathfrak{su}(2, \C)$) and thus it is in particular unbranched over $a = 0$ and $a = \infty$. Which of the above cases of constrained Willmore tori of revolution occur depends on whether the involution $\rho \circ \sigma$ of the CMC spectral curve has fixed points, which must lie over $a \in i \R.$ We show that these two different cases of CMC surfaces correspond to the distinction between orbitlike and wavelike profile curves.  Moreover, the different choices of the Sym-point $E$ used here to construct the curve correspond to the  different space forms in which the tori (respectively cylinders) have constant mean curvature. 
\begin{Pro}
Let $\gamma_E$ be an elastic curve in $H^2,$ as defined in Theorem \ref{curves} and $f$ the corresponding constrained Willmore cylinder of revolution.
Then $f$  is CMC in $H^3$ with mean curvature $|H|<1$ if and only if $\gamma_E$ is wavelike. 
If $\gamma$ is orbitlike we have the following:

For  $P_3(E) < 0$ the cylinder  $f$ is CMC in $S^3$. 

If $P_3(E) > 0$ $f$  is CMC in $H^3$ with mean curvature $H>1.$ 
\end{Pro}

\begin{proof}
The torus on which the Schwarzian derivative of the profile curve is defined is referred to in the following as the KdV spectral curve. It is an elliptic curve over the $E-$plane defined by the equation:
\begin{equation}\label{KdVspectralcurve} y^2= 4E^3 - g_2E - g_3.\end{equation} 
It can also be obtained by considering  the operator \begin{equation*}
D^E_2 = \del_x  +  \begin{pmatrix} 0 & q  - E -  \tfrac{1}{6}c \\ -1 & 0 \end{pmatrix},
\end{equation*}
where $q$ is the Schwarzian derivative of the curve and $c$ is as in lemma \ref{Lemma1} \footnote{Instead of the holonomies of $D^a_1$ over the $a-$plane, we consider the holonomies of $D^E_2$ over the $E-$plane in the above construction.}: This follows from the fact that a $\C^2-$function $(\psi_1, \psi_2)$ lies in the kernel of $D^E_2$ if and only if  $\psi_1 =  \psi_2'$ and $\psi_2$ solves the equation 
$$\psi_2'' + ( q - E -  \tfrac{1}{6}c) \psi_2 = 0.$$
\noindent
We first want to show how the CMC spectral curve of the surface and the KdV spectral curve of its profile curve are related. Let $$E =  - a^2  + \tfrac{1}{6}(\mu -1).$$ The equation defines a double covering of the $E-$plane by the $a-$plane branched at $E = \frac{1}{6} (\mu - 1)$ and $E = \infty.$ Further, we have $q = i\frac{\kappa'}{2}+ \frac{\kappa^2}{4} -  \frac{1}{4}.$ Then the gauge transformation from $D_2^E$ to $D_1^a$  is given by 
 \begin{equation*}
g = \begin{pmatrix}-i\frac{\kappa}{2} - ia  & - i\frac{\kappa}{2}  + ia \\1 &1\end{pmatrix},
\end{equation*}
for $a \in \C\setminus\{0\}$\footnote{The spectral curve is an analytic variety and it is thus determined by its generic points. }. 
This gauge defines a double covering $\tau$ of the
KdV spectral curve by the CMC spectral curve which is unbranched for $a \in \C \setminus\{0\}.$ Thus we only need to investigate what happens over $a= 0$ and $a= \infty.$
Since the CMC spectral curve is unbranched for these points and the parameter covering is branched, the covering of the spectral curves $\tau$ is unbranched if and only if $E= \infty$ and $E= \frac{1}{6} (\mu - 1)$ are branch points of the KdV spectral curve. This is the case by corollary \ref{g3}, since $\lambda = 0$ for constrained Willmore tori of revolution.\\

As mentioned before, a constrained Willmore torus of revolution is a CMC {\it torus} in a space form, if and only if the involution $\rho \circ \sigma$ has fixed points.
Since $\rho \circ \sigma$ interchanges the points over $a = \infty$ (see \cite{B, H1}), it has fixed points if and only if there are branch points of the CMC spectral curve over $a \in i\R.$ This happens if and only if the KdV spectral curve is branched over $E \in \R$ and $E > \frac{1}{6} (\mu - 1).$ Otherwise the torus is obtained through the glueing of two non compact CMC, $|H| < 1,$ cylinders in $H^3$ by \cite{BaBob}. \\

For wavelike elastic curves the KdV spectral curve has only $1$ real branch point over $E = \frac{1}{6} (\mu - 1)$ which vanishes over $a=0$. Therefore  there is no branch point of the CMC spectral curve over $a \in i\R.$\\ 

For orbitlike elastic curves the polynomial $P_3$ has $3$ real roots. By corollary \ref{g3} all roots are greater or equal to $\frac{1}{6}(\mu -1).$  Thus all branch points of the CMC-spectral curve lie over $a \in i\R$ and the involution $\rho \circ \sigma$ has fixpoints.  By the Sym-Bobenko formula, see \cite{Bob}, the surface is CMC in $S^3$ if the Sym-points are fixed under the involution $\rho \circ \sigma$ (which happens for $P_3(E) < 0$). If the Sym-points are no fix points of the involution ($P_3(E) > 0$), the surface is CMC in $H^3$.
\end{proof}
\begin{Rem}
A similar covering is given between the constrained Willmore spectral curve of a Hopf torus and  the KdV spectral curve of its spherical profile curve. In this case we have \begin{equation*}
D_1^a = \del_x
+    \begin{pmatrix}- ia & i \frac{\kappa}{2} - 1 \\ i\frac{\kappa}{2} + 1 & ia
         \end{pmatrix},
\end{equation*}
see \cite{H1}, the operator $D_2^E$ is defined as before but with $G= 1$ and the parameter covering is given by $E= -a^2 + \frac{1}{6}(\mu - 5).$ For $a \in \C \setminus\{0\}$ the gauge between the operators slightly changes and becomes
$$\tilde g = \begin{pmatrix}-i\frac{\kappa}{2} + 1 - ia  & -i\frac{\kappa}{2} - 1 + ia \\1 &1\end{pmatrix}.$$
As before the induced covering of the spectral curves is not branched for those points where the gauge is defined. Thus we need only to take a closer look at the points over $a= 0$ and $a= \infty.$ For $a= \infty$ the corresponding $E= \infty$ is a branch point of the $\wp-$function and the covering of the spectral curves is unbranched for these two points over $a= \infty$ as before.
But this does not hold for the points over $a = 0$ which corresponds to $E=  \frac{1}{6}(\mu - 5)$. By corollary \ref{g3} this is never a branch point of the KdV spectral curve. Hence the covering of the spectral curves is branched at the points over $a = 0$ and  by the Riemann-Hurwitz formula the constrained Willmore spectral curve of a Hopf torus has genus $2$. 
\end{Rem}
\subsection{Conformal Type and Willmore Energy}
The conformal types of tori of revolution and Hopf tori in terms of their profile curve were derived in \cite{LS1} and \cite{P}. 
\begin{The}
Let $f: T^2 \rightarrow S^3$ be either a constrained Willmore torus of revolution or a constrained Willmore Hopf torus determined by the formulas of Theorem \ref{curves} for fixed parameters $g_2, g_3, E \in \R$. Then we have the following.
\begin{itemize}
\item If $f$ is a constrained Willmore torus of revolution, then its conformal class is given by the lattice generated by $z_1 = 2\pi$ and  $z_2 = i \sqrt{G} L$ and its Willmore energy is 
$$\mathcal W (f) = 8 n \eta_1\pi - 4 n \omega_1 \wp(\omega_3) \pi.$$

\item If $f$ is a constrained Willmore Hopf torus, then its conformal class is given by the lattice generated by $z_1 = 2\pi$ and  $z_2 = \tfrac{1}{2}GA +  \tfrac{1}{2} i \sqrt{G} L  $ and its Willmore energy  is  
$$\mathcal W (f) = \tfrac{1}{\sqrt G}(16 n  \eta_1\pi - 8 n \omega_1 EÊ\pi).$$
\end{itemize}
Here $L = 2n\omega_1$ denotes the length of the curve in the respective space form and $A$ is the oriented enclosed area of the curve in $S^2$ is given by 
$$\tfrac{1}{2}GA  \text{ mod } 2 \pi=  \left( m \pi  - 4i n \eta_1 x_0  - 2n \omega_1 ( \tfrac{1}{2}\kappa(0)- 2i\zeta(x_0) \right) \text{ mod } 2\pi.$$
\end{The}

\begin{proof}
We first compute the Willmore energy of the tori.
Recall that for constrained Willmore tori of revolution and constrained Willmore Hopf tori  we have
$$ \wp(x+ x_0) = \tfrac{1}{4} i \kappa' - \tfrac{1}{8} \kappa^2 - \tfrac{1}{12}(\mu + G),$$
where $\kappa$ is the geodesic curvature of the arclength parametrized profile curve in the space form of curvature $G$. Thus the integral of the real part of the Weierstrass $\wp$ function, i.e., the real part of Weierstrass $\zeta-$function determines the bending energy of the curve. We have
\begin{equation*}
\begin{split}
\int_\gamma (\kappa^2 + \tfrac{2}{3}(\mu + G)) ds = 8 n (Re(\zeta(x-x_0 + 2\omega_1) - \zeta(x-x_0)) ) =  16 n \eta_1,
\end{split}
\end{equation*}
if the curve closes after $n$ periods of $\wp$. 
For constrained Willmore tori of revolution the Willmore energy is given by 
$$\mathcal W(f) =  \tfrac{1}{2}\pi  \int_\gamma \kappa^2 ds.$$
Since constrained Willmore tori of revolution comes from elastic curves in $H^2$, we have $\wp(\omega_3) = \tfrac{1}{6}(\mu + G)$
 and thus
$$ \mathcal W(f) = 8 n\eta_1  \pi - 4n\omega_1\wp(\omega_3)\pi.$$
\\

\noindent
For constrained Willmore Hopf tori  we have 
$$\mathcal W (f) =   \tfrac{1}{\sqrt G} \pi \int_\gamma (\kappa^2 + G) ds.$$
Since $E = \tfrac{1}{6}(\mu - \tfrac{1}{2}G)$ the Willmore energy of a constrained Willmore Hopf torus is computed to be
$$\mathcal W (f) = \tfrac{1}{\sqrt G}(16 n \eta_1 \pi- 8n \omega_1E \pi).$$

Now we turn to the conformal type of the tori considered. The conformal type is given by two vectors generating the lattice $\Gamma \in \C.$ In the case of tori of revolution these are given by
$$z_1 = 2 \pi \quad \text{ and } \quad z_2 = i \sqrt{G} L$$
where $L$ is the length of the curve in the space form of curvature $G <0$.
Since the profile curve $\gamma_E$ is arclength parametrized, we get
that the length of the curve is  $2 n \omega_1.$\\

\noindent 
For constrained Willmore Hopf tori the lattice is generated by
$$z_1 = 2\pi \quad \text{ and } \quad z_2 = \tfrac{1}{2}G A + \tfrac{1}{2}i \sqrt{G}L,$$
where $A$ is the oriented enclosed area of the curve, see \cite{P}, which is only well defined modulo $\tfrac{1}{G}4 \pi$.
By the Gau\ss-Bonnet theorem the enclosed area of a curve is given by  $$G A =  2\pi m -  \int_\gamma \kappa ds \text{ mod } 4 \pi,$$   where $m$ is the winding number of the curve. On the other hand we have: \newline Im$\zeta(x+x_0)  = \frac{\kappa}{4} -\frac{ \kappa_0 }{4}- i\zeta(x_0). $
Thus: 
\begin{equation*}
\begin{split}
\tfrac{1}{2}\int_\gamma \kappa ds - 2 n \omega_1 ( \tfrac{1}{2}\kappa(0)- 2i\zeta(x_0)) &=  2 \text{Im}\left( \ln(\sigma(x+x_0 + 2 n \omega_1)) - \ln( \sigma(x+ x_0)) \right )\\
&= - i \ln\left (\frac{e^{2  n \eta_1 (x+x_0 + \omega_1)}\sigma(x+x_0)\sigma(x-x_0) }{e^{2  n \eta_1 (x- x_0 + \omega_1)}\sigma(x+ x_0)\sigma(x- x_0)}\right)\\
&= -i \ln\left (e^{4 n \eta_1 x_0}\right).
\end{split}
\end{equation*}
The logarithm is only well defined modulus $2\pi i$. We obtain 
$$\tfrac{1}{2} \int_\gamma \kappa ds - 2 n \omega_1 ( \tfrac{1}{2}\kappa_0- 2i\zeta(x_0)) =  (2\pi - 4ni \eta_1 x_0 ) \text{ mod } 2 \pi.$$
Therefore $\tfrac{1}{2}G A$ is given by:
$$ \left(\pi m   - 4i n \eta_1 x_0  + 2n \omega_1( \tfrac{1}{2}\kappa_0- 2i\zeta(x_0) \right) \text{ mod } 2\pi  .$$

\end{proof}

\end{document}